\documentclass[reqno]{amsart}
\usepackage{enumitem}
\usepackage{color}
\usepackage{dsfont}
\usepackage{graphicx}

\def\CC{\mathbb{C}}

\def\R{\mathbb{ R}}
\def\N{\mathbb{ N}}
\def\T{\mathbb{T}}

\def\D{\mathcal{ D}}
\def\O{\mathcal{ O}}
\def\L{\mathcal{ L}}
\def\F{\mathcal{F}}

\def\J{\mathcal{J}}
\def\F{\mathcal{F}}

\def\eps{\varepsilon}
\def\sig{\sigma}

\def\i{\mathrm{i}}
\def\e{\mathrm{e}}
\def\d{\mathrm{d}}
\def\tr{\mathsf{T}}
\newtheorem{theorem}{Theorem}

\newtheorem{lemma}[theorem]{Lemma}

\theoremstyle{remark}
\newtheorem{remark}[theorem]{Remark}
\newtheorem{definition}[theorem]{Definition}
\newtheorem{assumption}[theorem]{Assumption}

\begin{document}

\title[Uniformly accurate time integrators]%
{High-order uniformly accurate time integrators for semilinear wave
equations of Klein--Gordon type in the non-relativistic limit}
\date{\today}

\author[H. Mohamad]{Haidar Mohamad}
\author[M. Oliver]{Marcel Oliver}
\address[H. Mohamad and M. Oliver]{Mathematical Institute for Machine
Learning and Data Science \\
KU Eichst\"att--Ingolstadt \\
85049 Ingolstadt \\ Germany}

\address[M. Oliver]{Constructor University, 28759 Bremen, Germany}


\begin{abstract}
We introduce a family of high-order time semi-discretizations for
semilinear wave equations of Klein--Gordon type with arbitrary smooth
nonlinerities that are uniformly accurate in the non-relativistic
limit where the speed of light goes to infinity.  Our schemes do not
require pre-computations that are specific to the nonlinearity and
have no restrictions in step size.  Instead, they rely upon a general
oscillatory quadrature rule developed in a previous paper (Mohamad and
Oliver, SIAM J. Num. Anal. 59, 2021, 2310--2319).
\end{abstract}

\maketitle
   
\section{Introduction}\label{Intro}

Let $X$ be a Hilbert space with inner product
$\langle \, \cdot \, , \, \cdot \, \rangle$ and associated norm
$\lVert \, \cdot \rVert$.  We study a semilinear wave equation on
$X$,
\begin{subequations}\label{e.MainEq}
\begin{gather}
  c^{-2} \, \partial_{tt} \phi + L \phi + c^2 \, \phi
  = f(\phi, t) \,, \\
  \phi(0) = \phi_0 \,, \\
  \partial_t \phi(0) = \phi'_0 \,,
\end{gather}
\end{subequations}
where $\phi \colon [0, T] \to X$, $L$ is a closed, densely defined,
self-adjoint, non-negative operator on $X$ with domain $\D(L)$, $c$ is
a positive constant, and $f \colon \D(L) \times [0,T] \to X$ a smooth
function.  Such equations arise, for example, in acoustics,
electromagnetics, quantum mechanics, and geophysical fluid dynamics,
both in the ``relativistic'' ($c=1$) and ``nonrelativistic''
($c\gg 1$) regimes.  The motivation for studying \eqref{e.MainEq} as
written is that it covers two well-studied special cases:
\begin{enumerate}[label={\upshape(\roman*)}]
\item $X = H^r(\T^d)$, $L = \Delta$, and
$f(\phi, t) = |\phi|^2 \, \phi$, which corresponds to the standard
semilinear Klein--Gordon equation.
\item $X = \R^{2d}$, $L =0$, and
$f(\phi, t) = -2 \, \e^{- c^2 t J} \, \nabla V( \e^{c^2 t J} \phi)$,
where $J$ is the canonical symplectic matrix in $2d$ dimensions and
$V$ is a smooth potential.  Changing variables via
$q = \e^{c^2 t J} \phi$, we can write the system in the standard form
\begin{subequations}\label{e.ODE}
\begin{gather}
  \dot q = p \,, \\
  \frac{1}{2 c^2} \, \dot p =  J p - \nabla V(q) \,.
\end{gather}
\end{subequations}
This system has been used as a finite-dimensional toy model for
rotating fluid flow, where the limit $c \to\infty$ corresponds to a
rapidly rotating earth.
\end{enumerate}
Analytically, the non-relativistic limit regime is well-studied for
these two examples.  We refer the reader to
\cite{FaouS:2014:AsymptoticPS, Masmoudi_KG} for the case of the
Klein--Gordon equation and to \cite{CotterReich06,
GottwaldMO:2017:OptimalBA, GottwaldO:2014:SlowDD} for the case of
system \eqref{e.ODE}.  Numerically, equation \eqref{e.MainEq} is
extensively studied in the relativistic regime
\cite{Gauckler:2015:ErAnlSWE, StraussVazquez:1978:NumKG}.  However,
due to the high oscillatory character of the solutions when $c$ is
large, most numerical methods suffer from severe time step restriction
in the non-relativistic regime.

Several authors have considered the problem of finding
``asymptotics-preserving numerical schemes'', i.e., schemes that
perform uniformly in this singular limit. Some of these schemes
\cite{FaouS:2014:AsymptoticPS} are based on a modulated Fourier
expansion of the exact solution
\cite{CohenHairerLubich:2003:ModFourierHighOscil,
HairerLubichWanner:2006:StrPresAlgo} where the highly oscillatory
problem in \eqref{e.ODE} is reduced to a non-oscillatory limit
Schr\"{o}dinger equation for which no time step restriction is needed.
Other schemes are based on multiscale expansions of the exact solution
\cite{BaoCaiZho:2014:MultiscaleNumKG,
ChartierCroMht:2015:UnifAccNumKG}.  Chartier \emph{et al.}\
\cite{ChartierLM:2020:NewCU} recently introduced a new method which
employs an averaging transformation to soften the stiffness of the
problem, hence allowing standard schemes to retain their order of
convergence.   Baumstark \emph{et al.}\ 
\cite{BaumstarkFS:2018:UniformlyAE} construct first and second order
uniformly accurate integrators for the Klein--Gordon equation with
\emph{cubic} nonlinearity by integrating the trigonometric products
arising from a suitable mild formulation explicitly.

In this paper, we develop a family of high-order
asymptotics-preserving schemes for \eqref{e.MainEq} that do not
require pre-computations tied to the specific nonlinearity $f$ and
have no restrictions in time step size.  The construction of the new
schemes is explained in Section \ref{Sec.AccuSchm}. We outline here
their main ingredients, where the first two follow the prior work
\cite{BaumstarkFS:2018:UniformlyAE}:
\begin{enumerate}[label={\upshape(\roman*)}]
\item Reformulate \eqref{e.MainEq} as a coupled first order
system using a linear transformation.
\item Factor out the rapidly rotating phase to make it explicit. 
\item Iterate the resulting mild formulation up to the desired order for the
coupled first order system in the new variables.
\item Use the quadrature rule developed in
\cite{MohamadO:2021:NumericalIF} to handle the high oscillatory integral in the resulting mild
formulation and complete the construction of the scheme.
\end{enumerate}
The remainder of the paper is structured as follows. In
Section~\ref{Sec.Prelim}, we state some properties of the operator $L$
within the framework of its functional calculus.  In
Section~\ref{Sec.Quad}, we introduce quadrature rules for the
approximation of highly oscillatory Banach-space-valued functions in
specific settings that will fit the construction of our
schemes. Section~\ref{Sec.AccuSchm} is devoted to the detailed
construction of the schemes.  Our main result on the order of
convergence of the schemes is stated and proved in
Section~\ref{Sec.ConvAnl}.  In Section~\ref{s.numerics}, we
demonstrate that the new schemes are accurate to their expected order
and that their error behavior is indeed uniform in $c$.

\section{Preliminaries}\label{Sec.Prelim}

We define the operators $B_c = c^{-1} \, \sqrt{L + c^2}$ and
$A_c = c^2 \, B_c - c^2$.  These operators are well-defined via the
spectral theorem for densely defined normal operators (e.g.\
\cite{Schmuedgen}).  Indeed, for any densely defined normal operator
$P\colon \D(P) \subseteq X \to X$, there exists a unique spectral
measure $E_P$ on the Borel $\sigma$-algebra $\mathcal{B}(\CC)$ into
the orthogonal projections on $X$ such that
\begin{equation}
  P = \int_{\CC} \lambda \, \d E_P(\lambda)
  = \int_{\sigma(P)} \lambda \, \d E_P(\lambda) \,.
\end{equation}
This integral representation of $P$ allows us to define the
assignments $P\mapsto f(P)$ for any $E$-a.e.\ finite measurable
function $f$ by the formula
\begin{equation}
  f(P) = \int_{\sigma(P)} f(\lambda) \, \d E_P(\lambda)
\end{equation}
with domain
\begin{equation}
  \D(f(P)) =
  \biggl\{
    x \in X \colon
    \int_{\sigma(P)} |f(\lambda)|^2 \,
    \d \langle E_P(\lambda)x, x \rangle < \infty
  \biggr\} \,.
  \label{e.domain}
\end{equation}
\begin{definition}
Let $A$, $B$ be two densely defined normal operators.  If
$\D(AB) \subseteq \D(BA)$ and $AB = BA$ on $\D(AB)$, we write
$AB \subseteq BA$ and say that ``$A$ commutes with $B$.''
\end{definition}
We fix in what follows an operator $J\in \L(X)$ such that 
\begin{equation}
  J L \subseteq L J \,, \quad
  J^* = -J \,, \quad
  \text{and} \quad
  J^2 = -I \,. 
\end{equation}
We now collect important elementary properties of the operators $J$,
$A_c$, and $B_c$.
\begin{lemma}\label{LemSpecTheo}
The operators $J$, $A_c$, and $B_c$ satisfy the following properties.
\begin{enumerate}[label={\upshape(\roman*)}]
\item\label{i} $\D(L) \subseteq \D(A_c) = \D(B_c) = \D(J A_c) = \D(J B_c)$,

\item\label{ii} $\|A_c u\|_{\D(L^j)}\leq \frac 12  \,
\|u\|_{\D(L^{j+1})}$ for any $j\in \N$, 

\item\label{iii} $J$ and $\e^{tJ}$ commute with $f(L)$ for any
measurable function $f\colon \R \to \R$ and $t\in \R$; in particular,
$J$ and $\e^{tJ}$ commute with $A_c$, $B_c$, and $B_c^{-1}$,

\item\label{iv} $\e^{ t J A_c}$ commutes with $J$ and $f(L)$ for any
measurable function $f \colon \R \to \R$ and $t\in \R$,

\item\label{v} $\| \e^{ t J A_c} \| \leq 1$, and

\item\label{vi} $\|(\e^{ t J A_c}- I)u\| \leq \frac 12 \, |t| \,
\|u\|_{\D(L)}$.
\end{enumerate}
\end{lemma}

\begin{proof}
The inclusion in \ref{i} follows directly from  
\begin{equation}
  \int_{\sigma(L)} |\lambda + c^2| \,
    \d \langle E_L(\lambda) u, u \rangle
  \leq (c^2 + \tfrac 12) \, \|u\|^2 + \tfrac 12 \int_{\sigma(L)}
    |\lambda |^2 \, \d \langle E_L(\lambda) u, u \rangle \,;
\end{equation}
the remaining identities are obvious. To prove \ref{ii}, we note that,
for $\lambda \geq 0$
\begin{equation}
  c\sqrt{\lambda + c^2} - c^2
  \leq \frac\lambda2 \,,
\end{equation}
and
\begin{equation}
  \| A_c u\|^2_{\D(L^j)}
  = \int_{\sigma(L)} \left| c \, \sqrt{\lambda + c^2} - c^2\right| ^2
    (1 + |\lambda |^2)^j \, \d \langle E_L(\lambda) u, u \rangle \,.
\end{equation}
For \ref{iii}, we recall that $J$ is bounded and commutes with $L$.
Thus, by \cite[ Proposition~5.15]{Schmuedgen},
$J \, E_L(K) = E_L(K) \, J$ for all $K \in
\mathcal{B}(\CC)$. Consequently,
$\e^{t J} \, E_L(K) = E_L(K) \, \e^{t J}$ for all
$K \in \mathcal{B}(\CC)$ and $t\in \R$.  Then the claim is a direct
consequence of \cite[Proposition~4.23]{Schmuedgen}.  For \ref{iv},
note that
\begin{equation}
  \e^{t J A_c}
  = \int_{\CC^2} \e^{c t \lambda (\sqrt{\mu ^2 +c^2}-c)} \,
    \d E_J(\lambda) \, \d E_L(\mu)
\end{equation}
where the integral is with respect to the product measure
$E_J \otimes E_L( K_1 \times K_2) = E_J(K_1) \, E_L(K_2)$ for all
$K_1, K_2 \in \mathcal{B}(\CC)$.  Hence, $E_J$ and $E_L$ commute with
$E_J \otimes E_L$ in the sense that
$E_J \otimes E_L( K_1\times K_2) \, E(K_3) = E(K_3) \, E_J \otimes E_L
(K_1\times K_2)$ for all $K_1, K_2, K_3 \in \mathcal{B}(\CC)$.  Once
again, the claim follows from \cite[Proposition~4.23]{Schmuedgen}.
Estimate \ref{v} is a direct consequence of the skew-symmetry of $J$.
Finally, to prove estimate \ref{vi}, let $u\in \D(L)$.  Since the
spectrum of $J$ is purely imaginary and $|\e^{\i x} - 1|^2 \leq x^2$
for $x\in \R$, we estimate
\begin{align}
  \|(\e^{ t J A_c}- I)u\|^2
  & = \int_{\sigma(J) \times \sigma(L)}
      |\e^{t c\lambda \bigl(\sqrt{\mu ^2 +c^2}-c\bigr)} - 1 |^2 \,
      \langle \d E_J(\lambda) \, \d E_L(\mu) u , u \rangle
    \notag \\
  & \leq t^2 \int_{\sigma(J) \times \sigma(L)}
      |c\lambda \, \bigl(\sqrt{\mu ^2 +c^2}-c \bigr) |^2 \,
      \langle \d E_J(\lambda) \, \d E_L(\mu) u , u \rangle
    \notag \\
  & \leq t^2 \, \| J A_c u\|^2 \,.
\end{align}
The claim then follows by estimate \ref{ii}.
\end{proof}

\begin{remark}
Lemma~\ref{LemSpecTheo}\ref{iii} and~\ref{iv} imply that if $P$ and
$Q$ are two operators such that $P$ is bounded and $ PQ \subseteq QP$,
then $\D(PQ) = \D(Q)$ and $P(\D(Q)) \subseteq \D(Q)$.  In other words,
the domain of $Q$ is invariant under any \emph{bounded} operator
commuting with $Q$.  In this paper, the analysis of the numerical
schemes assumes solutions of \eqref{e.MainEq} in $\D(L)$ which is, in
view of this remark, invariant under any bounded operator commuting
with $L$, in particular $J$, $\e^{t J}$, and $\e^{tJ A_c}$.
\end{remark}

\section{Quadrature for  Banach-space-valued functions}
\label{Sec.Quad}

In this section, let $(X, \|\cdot\|)$ be a complex Banach space and
$\Omega \subset \CC$ be open.  A function  $F \colon \Omega \to X$
is analytic if it is differentiable, i.e., provided for every $z_0 \in
\Omega$ there exists $F'(z_0) \in X$ such that
\begin{equation}
  F'(z_0) = \lim_{z\to z_0} \frac{F(z)- F(z_0)}{z- z_0} \,.
\end{equation}
The following simple lemma shows that estimates on the quadrature
error for differentiable complex-valued functions directly imply a
corresponding estimate for $X$-valued functions.

\begin{lemma}
\label{LemNumIntRuleBanach}
Let $I$ be an open interval on the real line and $\mu$ a measure on
$I$, possibly discrete.  Suppose a quadrature rule with nodes
$x_k \in I$ and weights $\omega_k$, $k=1,\dots,n$ satisfies the error
estimate
\begin{equation}
  \biggl\vert
    \int_I f(x) \, \d \mu (x)  - \sum_{k =1}^n \omega_k \, f(x_k)
  \biggr\vert
  \leq C(n,I) \, \sup_{x \in I} |f^{(p)}(x)| \,,
  \label{NumIntRule}
\end{equation}
for some $p \in \N$ and every $f \in C^p(I,\CC)$.
Then the quadrature rule satisfies the error estimate
\begin{equation}
  \biggl\Vert
    \int_I F(x) \, \d \mu(x) - \sum_{k =1}^n \omega_k \, F(x_k)
  \biggr\Vert
  \leq C(n,I) \, \sup_{x \in I} \Vert F^{(p)}(x) \rVert \,,
  \label{NumIntRuleBanach}
\end{equation}
where the integral is understood in the Bochner-sense, for every
$F \in C^p(I,X)$. 
\end{lemma}

\begin{proof}
Fix $\psi \in X^*$.  Let
\begin{equation}
  e_n = \int_I F(x) \, \d \mu(x) - \sum_{k =1}^n \omega_k \, F(x_k) \,.
\end{equation}
Due to the properties of the Bochner integral,
\begin{equation}
  \psi (e_n)
  = \int_a^b \psi \circ F (x) \, \d \mu
  - \sum_{k =1}^n \omega_k \, \psi \circ F (x_k) \,,
\end{equation}
so that, applying \eqref{NumIntRule} to $f = \psi \circ F$, we obtain
\begin{align}
  \lvert \psi (e_n) \rvert
  & \leq C(n,I) \, \sup_{x \in I } \,
           \bigl\vert \tfrac{\d^p}{\d x^p}[\psi \circ F](x) \bigr\vert
    \notag \\
  & \leq C(n,I) \, \lVert \psi \rVert_* \,
           \sup_{x\in I} \, \lVert F^{(p)}(x) \rVert \,.
\end{align}
By the Hahn--Banach theorem, we can choose $\psi \in X^*$ with
$\lVert \psi \rVert_* \leq 1$ such that
$\psi(e_n) = \lVert e_n \rVert$.  This implies
\eqref{NumIntRuleBanach}.
\end{proof}

With the help of this lemma, we lift three known estimates for the
quadrature error of complex-valued functions to the Banach space
setting.  The first concerns the trapezoidal rule approximation for
the integral of a $1$-periodic function $F$, namely the uniformly
weighted Riemann sum
\begin{equation}
  T_n(F) = \frac{1}{n} \sum_{k=0}^{n-1} F \Bigl( \frac kn \Bigr) \,.
\end{equation}
For given $a>0$, let
\begin{equation}
  \Omega_a = \{z \in \mathbb{C} \colon -a < \operatorname{Im}z < a \} \,.
\end{equation}
Then the following estimate, proved for $X=\CC$ in
\cite{TrefethenW:2014:ExponentiallyCT}, holds true.

\begin{theorem}\label{ThmTrapezBanach}
Let $F$ be an $X$-valued function, $1$-periodic on the real line,
analytic with $\|F(z)\| \leq A$ on the strip $\Omega_a$ for some
$a>0$. Then for any $n \in \N$,
\begin{equation}
  \label{estimOscilIntegBanach}
  \biggl\Vert
    \int_0^1 F(x) \, \d x - T_n(F)
  \biggr\Vert
  \leq  \frac{2A}{\e^{a n} - 1} \,.
\end{equation}
The constant $2$ is as small as possible. 
\end{theorem}

The second concerns the Gauss formula for the integral of a function
$f$ defined on the interval $[-1,1]$,
\begin{equation}
  \label{GaussQuad}
  G_m (f) = \sum_{k = 1}^m \omega_k \, f(\xi_k) \,,
\end{equation}
where the $\xi_k$ are the zeros of the Legendre polynomial $p_m$ of
degree $m$ and the weights are given by
\begin{equation}
  \omega_k = \frac{2}{(1-\xi_k^2) \, [p'_m (\xi_k)]^2} \,.
\end{equation}
For given $b>a$ and $\rho > \frac 12(b-a)$, let $E_\rho(a, b)$ denote
the ellipse with foci $a$, $b$ such that the lengths of its minor and
major semiaxes sum up to $\rho$.  Namely,
\begin{equation}
  E_\rho(a, b)
  = \bigl\{
      z \in \CC \colon
      z = \tfrac 12  \, (\rho \, \e^{\i \theta}
          + \tfrac 14 (b-a)^2\, \rho^{-1} \, \e^{-\i \theta}) + \tfrac 12 (a+b), 0 \leq \theta < 2 \pi 
    \bigr\} \,,
\end{equation}
and $\Sigma_\rho(a, b)$ the open region in $\CC$ bounded by $E_\rho(a, b).$

The formula \eqref{GaussQuad} can easily be written out for functions
defined on an arbitrary interval $[a,b]$ using the affine change of
variables
\begin{equation}
  \ell \colon
  \Sigma_{\frac{2\rho}{b-a}}(-1,1) \to \Sigma_{\rho}(a,b) \,,
  \quad \ell(x) = \frac{b-a}{2}(x+1) + a \,.
  \label{e.affine}
\end{equation}

\begin{theorem}\label{ThmGaussBanach}
Fix $k \in \N$, $\eps_0 >0$, and $\rho > \frac 12 \, (b-a)$.  Set
$\alpha = \eps_0\,\min\{0,a,b\}$ and $\beta=\eps_0\,\max\{0,a,b\}$.
Let $F \colon [\alpha, \beta]\times \Sigma_\rho(a, b) \to X$ be such
that $\zeta \mapsto F(\zeta, z)$ is $k$-times differentiable for any
$z\in \Sigma_\rho(a, b)$ and that
$z \mapsto \partial_1^i F(\zeta, z)$, where $\partial_1$ denotes the
partial derivative with respect to the first argument, is analytic on
$\Sigma_\rho(a, b) $ for $i = 0, \dots, k-1$ and any
$\zeta \in [\alpha, \beta]$ with
\begin{subequations}
\begin{gather}
  \max_{i} \sup_{[\alpha, \beta]\times\Sigma_\rho(a, b)} 
    \lVert \partial_1^i F \rVert \leq A_{\mathrm{an}} \,,
    \\
  \sup_{[\alpha, \beta]\times [a,b ] } \lVert \partial_1^k F \rVert
    \leq A_{\mathrm{dif}} \,.
\end{gather}
\end{subequations}
We abbreviate $f(x) = F(\eps x, x)$.  Then, for any $m \in \N$ and $\eps\in(0, \eps_0]$,
\begin{equation}
  \label{GaussEstimX1Banach}
  \biggl\Vert
    \int_a^b f(x) \, \d x - G_m(f)
  \biggr\Vert
  \leq \frac{16 \, A_{\mathrm{an}} \, \e^{\eps (b-a)} \, \rho^2}{(2\rho-b+a)} \,
       \biggl( \frac{b-a}{2\rho} \biggr)^{2m+1} 
  +  \frac{2 \, A_{\mathrm{dif}} \, (b-a)^{k+1} \, \eps^k}{k!} \,,
\end{equation}
where 
\begin{equation}
  G_m(f) = \frac{b-a}{2} \sum_{i=1}^m\omega_i \, f(\eta_i)
\end{equation}
with nodes $\eta_i = \ell(\xi_i)$.
\end{theorem}

\begin{proof}
Writing the Taylor series with respect to the first variable of $F$,
we find that for every $x\in [a, b]$ there exists
$\xi = \xi(x)\in [a, b]$ such that
\begin{equation}
  f(x) = \sum_{i=0}^{k-1} \frac{(x-a)^i \, \eps^i}{i!} \,
           \partial_1^i F(\eps a, x)
        + \frac{(x-a)^k \, \eps^k}{k!} \, \partial_1^k  F(\eps\xi, x) \,.
  \label{e.taylor}
\end{equation}
Thus, the following estimate, proved for $X=\CC$ in
\cite{ChawlaJ:1968:ErrorEG}, holds true for the quadrature formula
\eqref{GaussQuad} applied on each
$f_i(z) = (z-a)^i \, \partial_1^i F(\eps a, z)$, $i=0, \dots, k-1$,
which is analytic and bounded on $\Sigma_{\rho}(a, b)$:
\begin{multline}
  \biggl\Vert
    \int_a^b f_i(x) \, \d x - G_m(f_i)
  \biggr\Vert \\
  \leq \frac{16 \, \rho^2}{(2\rho-b+a)} \,
       \biggl( \frac{b-a}{2\rho} \biggr)^{2m+1} \,
       \sup_{z\in \Sigma_\rho(a, b)}
       \lVert (z-a)^i \, \partial_1^i F(\eps a, z) \rVert \,.
\end{multline}
This yields the first term on the right of \eqref{GaussEstimX1Banach}.
The Lagrange remainder in \eqref{e.taylor} is estimated independently
for the continuum integral over the interval $[a,b]$ and for the
discrete integral $G_m$, in both cases yielding the same contribution
to the second term on the right of \eqref{GaussEstimX1Banach}.
\end{proof}

The third concerns the Gauss formula for the discrete sum
$\sum_{j=0}^{N-1} F(x_j)$ on equidistant nodes
\begin{equation}
  x_j = -1 + \frac{2 j}{N-1} \,, \qquad 0\leq  j \leq N-1
\end{equation}
with
\begin{equation}\label{HaidarMarcelSum}
  \frac{2}{N} \sum_{j=0}^{N-1} F(x_j)
  \approx S_n(F)
  \equiv \sum_{k=1}^n w_{k,N} \, F(s_{k,N}),
\end{equation}
where the quadrature nodes $s_{k,N}$ are the zeros of the so-called
Gram polynomial $p_{n,N}$ of degree $n$.  Such polynomials are defined
by their orthonormality with respect to the discrete equidistant sum,
namely
\begin{equation}
  \sum_{j=0}^{N-1} p_{n, N}(x_j) \, p_{k, N}(x_j) = \delta_{nk} \,.
\end{equation}
The weights $w_{k,N}$ are given by
\begin{equation}
w_{k,N} = \frac{a_{n,N}}{a_{n-1,N}} \, \frac{2}{N\,p'_{n, N}(s_{k,N}) \, p_{n-1, N}(s_{k,N})} \,,
\end{equation}
where $a_{n,N}$ denotes the leading coefficient of $p_{n,N}$.  For a
detailed derivation and discussion, see
\cite{AreaDG:2014:ApproximateCS,AreaDG:2016:ApproximateCS,
MohamadO:2021:NumericalIF}. 

\begin{theorem}\label{ThmGaussSumBanach}
Fix $n \in \N$ and let $F\colon [a, b] \to X$ be a $2n$-times
differentiable function with $\lVert F^{(2n)} \rVert \leq A$ on
$[a,b]$.  Then
\begin{equation}
  \label{GaussSumEstimX1Banach}
  \biggl\Vert
    \frac{b-a}{N-1} \sum_{j=0}^{N-1} F(y_j)
    - \frac{N \, (b-a)}{2(N-1)} \, S_n(F)
  \biggr\Vert
  \leq \frac{16 \, A \, (b-a)^{2n+1} \, n!^4}{(2n+1) \, (2n)!^3} \,.
\end{equation}
Formula $S_n(f)$ is defined with nodes $r_{k, N} = \ell(s_{k,N})$ and
the equidistant summation points are given by $y_j = \ell(x_j)$, where
$\ell$ is the affine change of variable \eqref{e.affine}.
\end{theorem}

\begin{proof}
Assume first that $a = -1$ and $b =1$; the general case then follows
via the affine change of variable $\ell$.  Assume further that
$X = \R$.  The general case where $X$ is a complex Banach space
follows by applying Lemma~\ref{LemNumIntRuleBanach}.

Thus, let $H$ be the unique polynomial of degree $2n-1$ satisfying the
Hermite interpolation problem
\begin{equation}
  F(s_{k,N}) = H(s_{k,N}) \,, \quad
  F'(s_{k,N}) = H'(s_{k,N}) \,, \quad
  k = 1, \dots, n \,.
\end{equation}
By Rolle's theorem, for any $x\in [-1, 1]$ there exists
$s(x)\in [-1 ,1]$ such that
\begin{equation}\label{HermiteInterp.}
  F(x)- H(x) = \frac{F^{(2n)}(s)}{(2n)!} \, q^2_{n, N}(x) \,, 
\end{equation}
where $q_{n,N}$ is the polynomial
\begin{equation}
  q_{n, N}(x) = \prod_{k=1}^n (x-x_{k,N}) \,.
\end{equation}
Since \eqref{HaidarMarcelSum} is exact for all polynomials of degree
less than $2n-1$,
\begin{equation}
  \frac{2}{N} \sum_{j=0}^{N-1} F(x_j) = S_n(H) = S_n(F) \,.
\end{equation}
Thus, using \eqref{HermiteInterp.}, we estimate
\begin{align}
  \biggl\Vert\frac{2}{N-1} \sum_{j=0}^{N-1} F(x_j)
    - \frac{N}{N-1} S_n(F)\biggr\Vert
  & =  \frac{2}{N} \sum_{j=0}^{N-1} \lVert F(x_j)-H(x_j) \rVert
    \notag \\
  & \leq \frac{2A}{(N-1) \, (2n)!} \sum_{j=0}^{N-1} q^2_{n, N}(x_j) \,.
\end{align}
Note that ${\rm deg}(q_{n, N}) = n$ and $q_{n, N}$  has the same zeros as the
Gram polynomial $p_{n, N}$.  Hence,
\begin{equation}
  p_{n, N} = a_{n, N} \, q_{n, N} \,,
\end{equation}
where the constant $a_{n, N}$ is given by \cite{MohamadO:2021:NumericalIF}
\begin{equation}
  a_{n, N}= \sqrt{\frac{(2N + 1) \, (N- n -1)!}{(N + n)!}} \,
        \frac{(2n)! \, (N-1)^n}{2^n \, n!^2} \,.
\end{equation}
Since $p_{n, N}$ is normalized,
\begin{equation}
  \sum_{j=0}^{N-1} q^2_{n, N}(x_j)
  = \frac{1}{a^2_{n,N}}
  = \frac{(N + n)!}{(2n + 1) \, (N- n -1)! \, (N-1)^{2n}} \,
    \frac{2^n \, n!^4}{(2n)!^2} \,.
\end{equation}
For $N >n\geq 1$, we have
\begin{align*}
\frac{(N + n)!}{(N- n -1)! \, (N-1)^{2n+1}}
  & =\frac{(N+n)\, (N+n-1)\cdots(N-n)}{(N-1)\, (N-1)\cdots (N-1)}\\
  &\leq \left(1+ \frac 1 n \right) ^{2n+1} \\ 
  &\leq 8\,,
\end{align*}
which completes the proof. 
\end{proof}

In the next section, we will need to approximate a double integral of
a function of two variables where one of the integrals is continuous,
the other discrete.  The following lemma combines estimates
\eqref{GaussSumEstimX1Banach} and \eqref{GaussEstimX1Banach} in the
form required later.

\begin{lemma}\label{2DGauss}
Fix $\gamma\in (0,1)$ and $0< T_0 < \tau_0 <1$.  Given
$F \colon [0, \tau_0] \times [0, T_0] \times \Sigma_{1/(2\gamma)}(0,
1) \to X$, we write $G(s, x,T) = F(s, Tx, x)$ and consider $G$ as a
function on $[0, \tau_0] \times [0, 1]\times [0, T_0]$.  Suppose the
following is true.
\begin{enumerate}[label={\upshape(\roman*)}]
\item For every $(x, T)\in [0, 1]\times[0, T_0] $,
$s\mapsto G(s, x,T)$ is $2n$-times differentiable with
\begin{equation}
  \sup_{[0, \tau_0]\times[0, 1]\times[0,  T_0]} \lVert \partial_1^{2n}G \rVert
  \leq A \,.
\end{equation}
\item For every $s \in [0,\tau_0]$, $(\zeta, z) \mapsto F(s,\zeta, z)$
satisfies the assumptions of Theorem~\ref{ThmGaussBanach} replacing
$(k, \eps_0, \rho, a, b)$ there by $(2n, T_0, 1/(2\gamma), 0, 1)$
here, with bounds that are uniform with respect to $s\in [0, \tau_0]$.
\end{enumerate}
Then there exists a constant $C=C(F, \gamma, \tau_0, T_0,n)$ such that
for any $(\tau,T) \in (0, \tau_0]\times(0, T_0]$ and $(m,N) \in \N^2$
with $T = \tau/N$ and $m<N,$ we have
\begin{equation}\label{2DEstimGauss}
  \biggl\|
    T \sum_{j =0}^{N-1} \int_0^1 G(jT, x,T) \, \d x
    - \frac{\tau}4 \sum_{i = 1}^n \sum_{k = 1}^m
      w_{i,N} \, \omega_k \, G(r_{i,N}, \eta_k,T)
  \biggr\|
  \leq C \, \tau \, (\gamma^{2m} + \tau^{2n}) \,.
\end{equation}
\end{lemma}

\begin{proof}
Using Theorem~\ref{ThmGaussSumBanach}, we find that 
\begin{equation}
 T \sum_{j =0}^{N-1}  G(jT, x,T)  = \frac{\tau}2 \sum_{i = 1}^n 
      w_{i,N}  \, G(r_{i,N}, x,T)  + R(x,T) \,,
  \label{GaussQuad1stVar}
\end{equation}
where, in view of assumption (i), estimate
\eqref{GaussSumEstimX1Banach} implies that there exists a constant $C$
depending on
$\sup_{[0, \tau_0]\times[0, 1]\times [0, T_0]} \lVert \partial_1^{2n}G
\rVert$ and $n$ such that
\begin{equation}
  \lVert R(x,T) \rVert
  \leq C \, \tau^{2n+1} \,.
\end{equation}
Since $F(r_{i, N}, \cdot, \cdot)$ satisfies the assumption of
Theorem~\ref{ThmGaussBanach} on
$[0, T_0]\times \Sigma_{1/(2\gamma)}(0, 1)$ on each node $r_{i,N}$.
Thus, taking the integral of \eqref{GaussQuad1stVar} over $[0,1]$ and
using estimate \eqref{GaussEstimX1Banach}, we obtain
\eqref{2DEstimGauss}.
\end{proof}

\section{Uniformly accurate schemes}\label{Sec.AccuSchm}

Following \cite{BaumstarkFS:2018:UniformlyAE}, we introduce ``twisted
variables'' in which the linear operator in the equation is uniform as
$c \to \infty$.  The twisting technique was also used in an earlier
paper of Castella \emph{et al.}\ \cite{Fauo:2009:AveTecHiOsHamPb} who,
in a related context, developed an averaging technique for
highly-oscillatory Hamiltonian problems.  In a first change of
variables, we set
\begin{subequations}
\begin{gather}
  U = \phi - c^{-2} \, B_c^{-1} J \dot \phi \,, \\
  V = \phi + c^{-2} \, B_c^{-1} J \dot \phi \,.
\end{gather}
\end{subequations}
In terms of the variables $U$ and $V$, equation \eqref{e.MainEq} reads
\begin{subequations}
\label{e.MainEqUV}
\begin{gather}
  J \dot U = -c^2 \, B_c U
    + B_c^{-1} f \bigl( \tfrac 12 \, (U + V), t \bigr) \,, \\
  J \dot V = c^2 \, B_c V
    - B_c^{-1} f \bigl( \tfrac 12 \, (U + V), t \bigr) \,. 
\end{gather}
\end{subequations} 
As a second change of variables, we define
\begin{equation}
  u = \e^{-c^2 t J} \, U \,, \qquad
  v = \e^{c^2 t J} \, V \,.
\end{equation}
In terms of $u$ and $v$, system \eqref{e.MainEqUV} takes the form
\begin{subequations}
  \label{e.DuhamelMainODE2}
\begin{gather}
  \dot u = J A_c u - J B_c^{-1} \, \e^{-c^2 t J} \,
    f \bigl( \tfrac 12 \, (\e^{c^2 t J} u + \e^{-c^2 t J} v), t \bigr) \,, \\
  \dot v = - J A_c v + J B_c^{-1} \, \e^{c^2 t J} \,
    f \bigl( \tfrac 12 \, (\e^{c^2 t J} u + \e^{-c^2 t J} v), t \bigr) \,. 
\end{gather}  
\end{subequations}
We can write this system more compactly in terms of the vector-valued
functions $W = (U, V)^\tr$ and $w = (u, v)^\tr$.  Letting $A_c$ and
$B_c$ act diagonally on $\D(A_c) \times \D(A_c)$ and defining
\begin{subequations}
\begin{gather}
  \J = \begin{pmatrix} J & 0 \\ 0 & -J \end{pmatrix} \,, \\
  \F(W,t) = (-J, J)^\tr \, f\bigl( \tfrac 12 \, (U+V), t \bigr) \,,
\end{gather}
\end{subequations}
we can write
\begin{equation}
  \dot w = \J A_c w + B_c^{-1} \, \e^{-c^2t\J} \,
    \F \bigl(\e^{c^2t\J} w, t \bigr) \,.
\end{equation}
Let $\tau>0$ be the time step of the numerical scheme.  We write
$t_i = i \tau$ for $i = 0, 1, 2, \dots$ and apply the Duhamel formula,
so that
\begin{multline}
  w(t_i + \tau)
  = \e^{\tau \J A_c} \, w(t_i) \\
    + B_c^{-1} \int_0^\tau \e^{(\tau -s) \J A_c} \,
        \e^{-c^2(t_i + s)\J} \,
        \F \bigl( \e^{c^2(t_i + s)\J} \, w(t_i + s), s \bigr) \, \d s \,.
  \label{e.DuhamelMainODE3}
\end{multline}
Since we are free adapt the time $\tau$ of what is to emerge as the
numerical scheme, it is convenient to select $\tau$ as an integer
multiple of the fast period $T = {2\pi}/{c^2}$ so that $\tau = NT$ for
some $N \in \N$.  As $\e^{s \J} = \cos(s) I + \sin(s)\J$,
\begin{equation}
 \e^{\pm c^2 t_i\J } = \e^{2 \pi i N \J} = I
  \label{e.expidentity}
\end{equation}
whenever $i$ is integer.  Thus, such factors drop out of all
expressions further below, reducing the computational cost of the
scheme.

The two following assumptions on the nonlinearity $f$ and on the
solution of \eqref{e.DuhamelMainODE3} are required for the rigorous
analysis of convergence.

\begin{assumption}\label{assumpNonlin}
For given $n \in \N$ and $\mathcal{T}_0 >0,$ we assume that $f$
satisfies the following:
\begin{enumerate}[label={\upshape(\roman*)}]
\item $t \mapsto f(u, t)$ is $2n$-times  differentiable for any $u\in
\D(L^{2n})$,
\item\label{ass.ii} $x \mapsto f(\e^{2 \pi x J} u,t)$ has an analytic
extension to $\Sigma_{\frac{1}{2 \gamma}}(0,1)$ for some
$\gamma\in (0, 1)$ for any $t \in [0,\mathcal{T}_0]$.

\item\label{ass.iii} $f$ is Lipschitz with respect to the first
argument on bounded sets of $X$ with a constant uniform in
$t\in [0, \mathcal{T}_0]$.

\item\label{ass.iv} For any $t\in [0, \mathcal{T}_0]$,
$f(\cdot, t) \colon \D(L^{2n})\to X$ is $2n$-times G\^ateau
differentiable such that
$D^kf(u, t) \in \L( \D(L^{2n-\alpha_k}), D(L^{2n-|\alpha_k|}))$ for
every $k = 1, \dots, 2n$, $u\in \D(L^{2n})$, and multi-index
$\alpha_k=(j_1,\dots, j_k)$ for which each component is larger than
$1$ and $|\alpha_k|\leq 2n $.
\end{enumerate}
Here, $\D(L^{2n- \alpha_k})$ refers to the direct product
$\D(L^{2n-j_1}) \times \cdots \times\D(L^{2n-j_k})$.
\end{assumption}

\begin{remark}
The nonlinearity of the semilinear Klein--Gordon equation introduced
in Section~\ref{Intro} satisfies Assumption~\ref{assumpNonlin}.
\end{remark}

\begin{remark}
To see how the differentiability requirement in
Assumption~\ref{assumpNonlin} \ref{ass.iv} arises, consider the
following example, which is a simplified version of the estimates
which arise in the analysis of the numerical scheme below.  Take
$g(s) = \e^{x J A_c} f(h(s))$, $h(s) = \e^{x J A_c} u$,
$u \in \D(L^n)$ and $n =1$.  Since
\begin{align}
  \e^{- x J A_c} \, g''(s)
  & = - A_c^2 \, f( h(s))
      + 2 J A_c \, D f(h(s)) \, h'(s) \notag \\
  & \quad
      + D^2 f(h(s))[ h'(s), h'(s)]
      + D f(h(s)) \, h''(s) \,, 
\end{align}
$\lVert g'' \rVert_X$ is uniformly bounded in $c$ provided
\begin{subequations}
\begin{gather}
  Df(u) \in\L( \D(L^{2n-1}), D(L^{2n-1})) = \L(\D(L)) \,, \\
  Df(u)  \in \L( \D(L^{2n-2}), D(L^{2n-2}))= \L(X) \,, \\
  D^2f(u)  \in \L( \D(L^{2n-(1,1)}, \D(L^{2n-2} )
  = \L(\D(L)\times \D(L), X) \,.
\end{gather}
\end{subequations}
This suffices to satisfy condition \ref{iv} of Lemma~\ref{2DGauss} for
$G(s,x,T)=g(s)$.
\end{remark}

\begin{assumption}\label{assumplocBound}
For given $n$, in the setting of Assumption~\ref{assumpNonlin}, there
exists $\mathcal{T} \in (0, \mathcal{T}_0]$ and $K>0$ independent of
$c$ such that
\begin{equation}
  \sup_{0\leq t \leq \mathcal{T} } \lVert w(t) \rVert_{\D(L^n)}
  \leq K \,.
\end{equation}
\end{assumption}
To guarantee uniform convergence with respect to $c$, we make the
following important observation which effectively asserts that the
time derivative $\dot w$ is bounded uniformly in $c$.
\begin{lemma}\label{lemBoundDeriv}
The solution $w$ of \eqref{e.DuhamelMainODE3} satisfies
\begin{equation}\label{EstimBoundDeriv }
  \lVert w(t_i + s) - w(t_i) \rVert
  \leq \frac{s}2 \, \lVert w(t_i) \rVert_{\D(L)}
  + s \, \sup_{\sigma \in [0, s]} \,
    \bigl\Vert
      \F \bigl(\e^{c^2(t_i + \xi)\J} \, w(t_i +\sigma) \bigr)
    \bigr\Vert \,.
\end{equation}
\end{lemma}

\begin{proof}
The proof is a direct application of estimate \ref{vi} in
Lemma~\ref{LemSpecTheo} and the fact that
$\lVert B_c^{-1} \rVert \leq 1$.
\end{proof}

In a first step, we define a sequence of ``pre-schemes''
$\Phi_l \colon X \times \R \to X$ which provide consistent
approximations to the right hand side of the Duhamel formula
\eqref{e.DuhamelMainODE3} to order $\tau^{l+1}$, namely
\begin{subequations}
  \label{SchemeSeq}
\begin{gather}
  \Phi_1(w, z)
  = \e^{z \J A_c} \, w - B_c^{-1} \int_0^z \, \e^{-c^2 s\J} \,
        \F (\e^{c^2 s\J} \, w, s) \, \d s \,,
  \label{SchemeSeq.a} \\
  \Phi_{l+1}(w, z)
  = \e^{z \J A_c} \, w - B_c^{-1} \int_0^z \e^{(z -s) \J A_c} \,
        \e^{-c^2 s\J} \,
        \F \bigl( \e^{c^2 s\J} \, \Phi_l(w, s), s \bigr) \, \d s \,.
  \label{SchemeSeq.b}
\end{gather}
\end{subequations}
The pre-schemes approximate the true solution in the following sense.

\begin{lemma}\label{LemRl}
Under Assumption~\ref{assumpNonlin} \ref{ass.iii}, let $w$ be a
solution for \eqref{e.DuhamelMainODE3} satisfying
Assumption~\ref{assumplocBound} for $n=1$, and fix $l \in \N^*$.  Then
there exist constants $C_l$ independent of $c$ such that all
$s \geq 0$,
\begin{equation}
  \lVert w(t_i + s) - \Phi_l(w(t_i), s) \rVert \leq C_l \, s^{l+1} \,.
\end{equation}
\end{lemma}

\begin{proof}
We set $R_l(w(t_i), s) = w(t_i + s) - \Phi_l(w(t_i), s)$ and proceed
by induction.  When $l=1$,
\begin{align}
  R_1(w(t_i), s)
  & = B_c^{-1} \int_0^s \, \e^{-c^2 \sig\J} \,
        \F \bigl( \e^{c^2 \sig\J} \, w(t_i), \sigma \bigr) \, \d \sigma
      \notag\\
  & \quad
    - B_c^{-1} \int_0^s \e^{(s -\sig) \J A_c} \, \e^{-c^2 \sig\J} \,
        \F \bigl( \e^{c^2 \sig\J} \, w(t_i +\sig), \sigma \bigr) \,
        \d \sigma \,.
  \label{R1}
\end{align}
The estimate on $R_1$ follows by using Lemma~\ref{lemBoundDeriv} to
freeze $w(t_i + \sigma)$ and Lemma~\ref{LemSpecTheo}(iv) to remove the
operator $\e^{(s-\sig ) \J A_c}$ in the second integral in \eqref{R1}.
For $l\geq1$,
\begin{align}
  R_{l+1}(w(t_i), s)
  & = B_c^{-1} \int_0^s \e^{(s -\sig) \J A_c} \, \e^{-c^2 \sig\J} \,
        \F \bigl( \e^{c^2 \sig\J} \,
          \Phi_l( w(t_i ), \sigma), \sigma \bigr) \, \d \sig
      \notag\\
  & \quad
    - B_c^{-1} \int_0^s \e^{(s -\sig) \J A_c} \,
        \e^{-c^2 \sig\J} \,
        \F \bigl( \e^{c^2 \sig\J} \, w(t_i +\sig), \sigma \bigr) \,
        \d \sigma \,.
\end{align} 
By Lemma~\ref{LemSpecTheo} and the fact that $f$ is Lipschitz on
bounded sets of $X$ with respect to the first argument, there exists a
constant $C$ independent of $c$ such that
\begin{equation}
  \lVert R_{l+1}(w(t_i), s)  \rVert
  \leq C \, s \, \sup_{\sig \leq s} \,
       \lVert R_l(w(t_i), \sig) \rVert \,.
\end{equation}
This completes the proof.
\end{proof}

While the operator $A_c$ and the associated semi-group $\e^{t\J A_c}$
are uniformly well-behaved as $c \to \infty$, the integrals in
\eqref{SchemeSeq} still contain highly oscillatory terms with a
\emph{single} fast frequency.  For the latter, effective numerical
quadrature is possible \cite{MohamadO:2021:NumericalIF}.  Following
the strategy developed there, we split $z/T \equiv N_z + \theta_z$
into its integer part $N_z = \lfloor z/T \rfloor$ and fractional part
$\theta_z = z/T - N_z$.  Then the integral in \eqref{SchemeSeq.a} can
be written
\begin{align}
  B_c^{-1} \int_0^z \, & \e^{-c^2 s\J} \,
    \F (\e^{c^2 s\J} \, w, s) \, \d s
    \notag \\
  & = B_c^{-1} \sum_{j=0}^{N_z-1} \int_{jT}^{(j+1)T}
        \e^{-c^2 s\J} \,
        \F (\e^{c^2 s\J} \, w, s) \, \d s
    \notag \\
  & \quad
      + B_c^{-1} \int_{N_zT}^{z} 
        \e^{-c^2 s\J} \,
        \F (\e^{c^2 s\J} \, w, s) \, \d s 
    \notag \\
  & = T \sum_{j=0}^{N_z-1} \int_0^1 G_0(jT,\sigma) \, \d \sigma
      + T \int_0^{\theta_z} G_0(N_zT,\sigma) \, \d \sigma
  \label{e.integral-rewrite}
\end{align}
with
\begin{equation}
  G_0(\rho,\sigma)
  = B_c^{-1} \, \e^{-2 \pi \sigma \J} \, \F (\e^{2 \pi \sigma \J} \, w,
    \rho+\sigma T)
  \label{e.G0}
\end{equation}
and where, in the second equality of \eqref{e.integral-rewrite}, we
have used \eqref{e.expidentity}.  Analogously, the integral in
\eqref{SchemeSeq.b} can be written
\begin{align}
  B_c^{-1} \int_0^z & \e^{-s \J A_c} \,
        \e^{-c^2 s\J} \,
        \F \bigl( \e^{c^2 s\J} \, \Phi_l(w, s), s \bigr) \, \d s
  \notag \\
  & = T \sum_{j=0}^{N_z-1} \int_0^1 G[\Phi_l](jT,\sigma) \, \d \sigma
      + T \int_0^{\theta_z} G[\Phi_l](N_zT,\sigma) \, \d \sigma \,,
\end{align}
where, for $\Upsilon \colon X\times \R \to X$,
\begin{equation}
  G[\Upsilon](\rho,\sigma)
  = B_c^{-1} \, \e^{-(\rho+\sigma T)\J A_c} \, \e^{-2 \pi \sigma \J} \,
    \F \bigl(\e^{2 \pi \sigma \J} \,
      \Upsilon(w, \rho+\sigma T), \rho+\sigma T \bigr) \,.
  \label{E.Gups}
\end{equation}
Altogether, \eqref{SchemeSeq} then takes the form
\begin{subequations}
  \label{SchemeSeq2}
\begin{gather}
  \Phi_1(w, z)
  = \e^{z \J A_c} \, w
    - T \sum_{j=0}^{N_z-1} \int_0^1 G_0(jT,\sigma) \, \d \sigma
    - T \int_0^{\theta_z} G_0(N_zT,\sigma) \, \d \sigma \,, 
  \label{SchemeSeq2.a} \\
  \Phi_{l+1}(w, z)
  = \e^{z \J A_c} \,
    \biggl( w
      - T \sum_{j=0}^{N_z-1} \int_0^1 G[\Phi_l](jT,\sigma) \, \d \sigma
      - T \int_0^{\theta_z} G[\Phi_l](N_zT,\sigma) \, \d \sigma
    \biggr) \,.      
  \label{SchemeSeq2.b}
\end{gather}
\end{subequations}
We now use the approximate the integrals in \eqref{SchemeSeq2} by
classical Gauss quadrature and the sums by Gauss summation to obtain
\begin{subequations}
  \label{SchemeSeq3}
\begin{align}
  \Psi_1(w, z)
  & = \e^{z \J A_c} \, w
    -\frac{N_z T}{4}\sum_{j=1}^n \sum_{k=0}^m w_{j,N_z} \,
      \omega_k \, G_0(r_{j,N_z}, \eta_k)\notag \\
  & \quad 
    - \frac{\theta_z T}{2} 
      \sum_{k=0}^m \omega_k \, G_0(N_zT,\theta_z\eta_k) \,,
  \label{SchemeSeq3.a} \\
  \Psi_{l+1}(w, z)
  & = \e^{z \J A_c} \,
    \biggl( w 
       - \frac{N_z T}{4} \sum_{j=1}^n \sum_{k=0}^m w_{j,N_z} \,
      \omega_k \, G[\Psi_l](r_{j,N_z}, \eta_k) \notag \\
  & \quad 
      - \frac{\theta_z T}{2}
         \sum_{k=0}^m \omega_k \, G[\Psi_l](N_zT,\theta_z\eta_k) \, 
    \biggr) \,. 
  \label{SchemeSeq3.b}
\end{align}
\end{subequations}

\begin{remark}
This last approximation is precisely the setting in which
Lemma~\ref{2DGauss} applies.  The function $F$ in the lemma
corresponds to the right hand sides of \eqref{e.G0} resp.\
\eqref{E.Gups}.  Then the function $G$ defined in the lemma coincides
with $G$ defined here, except that we have dropped the third argument
$T$ throughout this section as it is considered constant here in order
not to clutter up notation.
\end{remark}

\begin{remark} \label{r.tau}
For a scheme of global order $l$, we use $\Psi_l$ with $z=\tau$ as the
time stepper.  At the top level, the second sum in
\eqref{SchemeSeq3.a} or \eqref{SchemeSeq3.b} does not contribute.
However, when $l\geq 2$, the inner evaluations of
$\Psi_{l-1}, \Psi_{l-2}, \dots$ will generally be evaluated at points
$z$ that are not integer multiples of $T$, so that their $z$-arguments
have to be re-split into the respective integer ($N_z$) and fractional
($\theta_z$) multiples of $T$.  Thus, in general, the second sum on
the right of \eqref{SchemeSeq3} is required for consistency.
\end{remark}

\begin{remark} 
Note that in the case where $\F$ is constant with respect to the
second variable, the function $G_0 = G_0(x)$ is one-variable periodic
function. Thus, the approximation from Theorem~\ref{ThmTrapezBanach}
can also be used to define a first order scheme ($l=1$) with accuracy
that is exponential in the number of nodes.  More specifically, for
$\tau = NT$,
\begin{align}
  \Phi_1(w, \tau)
  & = \e^{\tau \J A_c} \, w  - \tau  \int_0^1 G_0(x) \, \d x
      \notag \\
  & = \e^{\tau \J A_c} \, w
      - \frac\tau m \sum_{k=0}^{m-1} G_0 \Bigl(\frac km \Bigr)
      + \O (\tau \,\e^{-d m}) 
\end{align}
for some $d>0$. 
\end{remark}

\begin{lemma}\label{LemSl}
Let $l, n \in \N^*$ and $w\in \D(L^{2n})$.  Fix $0<z_0<1$, $c_0>0$,
and assume that $f$ satisfies Assumption~\ref{assumpNonlin}, with
analyticity property \ref{ass.ii} valid on the ellipse
$\Sigma_{1/(2\gamma)}(0, 1)$ for some $\gamma \in (0,1)$.  Then there
exists $C_l=C_l(f, \|w\|_{\D(L^{2n})}, c_0, z_0, n)$ such that for all
$m \in \N^*$ and $z\leq z_0 <1$,
\begin{equation}
  \label{e.LemSl}
  \lVert \Psi_l(w, z) - \Phi_l(w, z) \rVert 
  \leq C_l \, z \, (z^{2n} + \gamma^{2m}) \,.
\end{equation}
\end{lemma}

\begin{proof} 
In view of the expression for each $G_0$
and $G$, there exist $F_0$ and $F$ such that 
\begin{equation}
  G_0(\rho, \gamma, T) = F_0(\rho, T\sigma, \sigma) \,, \qquad
  G(\rho, \gamma, T) = F(\rho, T\sigma, \sigma) \,,
\end{equation}
where, since $w\in \D(L^{2n})$ and
$f$ satisfies Assumption~\ref{assumpNonlin}, $F_0$ and
$F$ satisfy the conditions of Lemma~\ref{2DGauss} on $[0, z_0]\times
[0, 2\pi/c_0^2]\times\Sigma_{1/(2\gamma)}(0,
1)$.

We set $S_l(w, z) =\Psi_l(w, z) - \Phi_l(w, z)$ and proceed by
induction.  For $l=1$, we can directly use Lemma~\ref{2DGauss} for the
difference of first terms and Theorem~\ref{ThmGaussBanach} for the
difference of second terms, \eqref{e.LemSl} holds true as stated.  For
$l>1$, we have
\begin{align}
& \e^{-z \J A_c} \, S_{l+1}(w, z)
  \notag \\
  & = 
       - \frac{N_z T}{4} \sum_{j=1}^n \sum_{k=0}^m w_{j,N_z} \,
      \omega_k \, G[\Psi_l](r_{j,N_z}, \eta_k) \, 
      - \frac{\theta_z T}{2} \sum_{k=0}^m \omega_k \, G[\Psi_l](N_zT,\theta_z\eta_k) \, 
    \notag \\
    &  \quad
    + \frac{N_z T}{4}\sum_{j=1}^n \sum_{k=0}^m w_{j,N_z} \,
      \omega_k \, G[\Phi_l](r_{j,N_z}, \eta_k) \, 
      +\frac{\theta_z T}{2} \sum_{k=0}^m \omega_k \, G[\Phi_l](N_zT,\theta_z\eta_k) \, 
    \notag \\
    & \quad
    - \frac{N_z T}{4} \sum_{j=1}^n \sum_{k=0}^m w_{j,N_z} \,
      \omega_k \, G[\Phi_l](r_{j,N_z}, \eta_k) \, 
      - \frac{\theta_z T}{2} \sum_{k=0}^m \omega_k \, G[\Phi_l](N_zT,\theta_z\eta_k) \, 
    \notag \\
   & \quad + T \sum_{j=0}^{N_z-1} \int_0^1 G[\Phi_l](jT,\sigma) \, \d \sigma
      + T \int_0^{\theta_z} G[\Phi_l](N_zT,\sigma) \, \d \sigma
  \label{SlRemainder}
\end{align}
We write $S_{l+1}^{(1)}(w, z)$ and $S_{l+1}^{(2)}(w, z)$ to denote the
first two and the last two lines on the right of \eqref{SlRemainder},
respectively.  As $f$ is Lipschitz with respect to the first argument
on bounded sets of $X$, there exist $K_1$ and $K_2$, each depending on
$f$, such that
\begin{align}
  \|S_{l+1}^{(1)}(w, z)\|
  & \leq K_1 \,  z \, \sup_{j,k} \|S_l(w, r_{j,N_z} +\eta_k T)\| 
    + K_2 \,  z  \, \sup_k \|S_l(w, N_zT + \theta_z\eta_k T)\|
    \notag \\
  & \leq (K_1 + K_2) \, C_l \,  z^2 \, (z^{2n} + \gamma^{2m}) \,.
\label{SlRemainder1}
\end{align}
On the other hand, using Lemma~\ref{2DGauss}, there exists a constant
$K_3$ depending on $f$, $\|w\|_{\D(L^{2n})}$, $c_0$, $z_0$, 
and $n$ such that
\begin{equation}\label{SlRemainder2}
\|S_{l+1}^{(2)}(w, z)\| \leq K_3 \, z \, (z^{2n} + \gamma^{2m}) \,. 
\end{equation}
Thus, combining \eqref{SlRemainder1} and \eqref{SlRemainder2}, we
conclude that there exists a constant $C_{l+1}$ depending on $f$,
$\|w\|_{\D(L^{2n})}$, $c_0$, $z_0$  and $n$ such that
\begin{equation}
  \|S_{l+1}(w, z)\| \leq C_{l+1} \, z \, (z^{2n} + \gamma^{2m}) \,,
\end{equation}
which concludes the proof.
\end{proof}

As stated before, we select the time step $\tau$ to be an integer
multiple of the fast period $T$ so that $\tau = N T$ for some
$N \in \N$.  As a numerical approximation to the exact solution $w$ at
time $t_{i+1}$, we take the scheme
 \begin{subequations}
  \label{lthOrdScheme}
\begin{gather}
 w_{i+1}  =\Psi_l (w_i, \tau)   \,, \\
   w_0 =  \left(\begin{matrix} \phi_0 \\ \phi_0 \end{matrix}\right) -c^{-2}\J B^{-1}_c \left(\begin{matrix} \phi'_0 \\ \phi'_0 \end{matrix}\right).
\end{gather}  
\end{subequations}

\section{Convergence analysis}\label{Sec.ConvAnl}

The scheme \eqref{lthOrdScheme} satisfies the following global
estimate.

\begin{theorem}
Let $f$ satisfies Assumption \ref{assumpNonlin}, with analyticity
property (i) valid on the ellipse $\Sigma_{1/(2\gamma)}(0, 1)$ for
some $\gamma \in (0,1)$.  Fix $l\in \N^*$, $c_0 >0$, and let
$n = \lfloor \frac{l+ 1}{2} \rfloor$. Assume further that there exists
$\mathcal K>0$ such that for every $c \geq c_0$,
\begin{equation}\label{InitialEstim}
  \|\phi_0\|_{\D(L^{2n})} + c^{-2} \, \| B^{-1}_c \phi'_0\|_{\D(L^{2n})}
  \leq \mathcal K \,.
\end{equation} 
Then there exist $\mathcal{T} >0$ and
$C = C(f, \mathcal K, \mathcal{T}, c_0,  n)$ such that for all
$c \geq c_0$, $\tau \in \frac{ 2 \pi}{c^2} \N$,
$t_i \leq \mathcal{T}$, and $m \in \N^\ast$,
\begin{equation}
  \|\phi_i -\phi(t_i)\|  \leq C \, (\tau^l + \gamma^{2m})
  \label{e.global}
\end{equation}
where $\phi$ solves \eqref{e.MainEq} and
\begin{equation}
  \phi_i = \frac{(w_i)_1+ (w_i)_2}2
\end{equation}
with $w_i$ given by \eqref{lthOrdScheme}.
\end{theorem}

\begin{proof}
Note  first that for every $c \geq c_0$,
\begin{equation}
  \|w_0\|_{\D(L^{2n})}
  \leq \|\phi_0\|_{\D(L^{2n})}
       + \|c^{-2} \, J B^{-1}_c \phi'_0\|_{\D(L^{2n})}
  \leq \mathcal K \,.
\end{equation}
Thus, there exist two constants $\mathcal{T}, K >0 $ depending on
$c_0$ and $w_0$ for which Assumption~\ref{assumplocBound} is
satisfied.  Lemmas~\ref{LemSl} and~\ref{LemRl} allow us to write
\begin{align}
  w(t_i + \tau)
  & = \Phi_l(w(t_i), \tau) + R_l(w(t_i), \tau)
      \notag \\
  & = \Psi_l(w(t_i), \tau) + R_l(w(t_i), \tau) - S_l(w(t_i), \tau) \,.
\end{align}
Setting $e_i = \|w(t_i) - w_i\|$, we now split the error as follows:
\begin{align}
  e_{i+1}
  & \leq \| R_l(w(t_i), \tau)\| + \|S_l(w(t_i), \tau)\|
         \notag \\
  & \quad + \| \Psi_l(w(t_i), \tau) - \Psi_l(w_i, \tau)\| \,.
  \label{locErrlthOrd}
\end{align} 
Recalling that $\F$ is Lipschitz on $X$ and arguing by induction on
$l$, we find that there exists a constant $C_1 >0$ depending on $f$
such that
\begin{equation}\label{locEstim2}
  \|\Psi_l(w(t_i), \tau) - \Psi_l(w_i, \tau)\|
  \leq (1 +C_1 \tau)^l \, e_i \,.
\end{equation} 
By Lemma~\ref{LemRl} and~\ref{LemSl}, there exists a constant $C_2>0$
depending on $f$, $\mathcal K$, $\mathcal{T}$, $c_0$, and
$n$ such that
\begin{equation}\label{locEstim1}
  \|R_l(w(t_i), \tau)\| + \|S_l(w(t_i), \tau)\|
  \leq C_2 \, \tau \, (\tau^l + \gamma^{2m}) \,.
\end{equation} 
Then, \eqref{locErrlthOrd} reads 
\begin{equation}
  e_{i+1} \leq (1 + C_1\tau)^l \, e_i
               + C_2 \, \tau \, (\tau^l + \gamma^{2m}) \,.
\end{equation}
Thus, we find by induction that
\begin{equation}
  e_i \leq (1+ C_1 \, \tau)^{il} \, e_0
  + C_2 \, \frac{(1+ C_1 \, \tau)^{il} - 1}{C_1} \,
    (\tau^l + \gamma^{2m}) \,.
\end{equation}
Since $e_0 = 0$ and $1+ x \leq \e^x$, we obtain
\begin{equation}
  e_i \leq C_2 \, \frac{\e^{C_1 l \mathcal{T}} - 1}{C_1} \,
           (\tau^l + \gamma^{2m})
  \equiv C \, (\tau^l + \gamma^{2m}) \,.
  \label{e.global1}
\end{equation}
To obtain the final estimate, we undo the variable twist, noting that
\begin{equation}
  \phi(t_i)
  = \frac{(\e^{c^2 t_i J} (w(t_i))_1 + (\e^{-c^2 t_i J}(w(t_i))_2}2
  = \frac{w(t_i)_1 + w(t_i)_2}2 \,.
\end{equation}
Then \eqref{e.global1} directly implies estimate \eqref{e.global}.
\end{proof}

\section{Numerical tests}
\label{s.numerics}

We now demonstrate the scaling behavior of the new uniformly accurate
time integrators (UAT) in a simple test case where an explicit
reference solution is available.  Our example has
$\phi \colon [0, T] \times \T^d \to \CC$ with $L = \delta - \Delta$
and $f(\phi) = |\phi|^2 \phi$ with some $\delta >0$.  Then, for
arbitrary $a \in \R^d$, the function
\begin{equation}
  \phi(t, x) = \sqrt{\delta + |a|^2} \, \e^{\i(c t + a \cdot x)}
\end{equation}
is a solution of \eqref{e.MainEq} with
\begin{equation}
  \phi_0 = \sqrt{\delta + |a|^2} \, \e^{\i a \cdot x} \,, \quad
  \phi'_0 = \i c\sqrt{\delta + |a|^2} \, \e^{\i a \cdot x} \,.
\end{equation}
For simplicity, we consider only solutions with no dependence on $x$,
i.e.\ $a =0$, where
\begin{equation}
  \phi(t, x) = \phi(t) = \sqrt{\delta } \, \e^{\i c t } \,.
\end{equation}
The order of the Gaussian quadrature approximating the
inner integral in \eqref{SchemeSeq2} is chosen as $m=6,8,10$ at level
$l=0,1,2$.  Theoretically, in view of \eqref{e.global},  $m$ should be
chosen so that
\begin{equation}
  m \approx \frac{\ln(\tau)}{2 \ln(\gamma)} \, l \,.
\end{equation}
However, as we do not have any access to a good estimate for $\gamma$,
we determined a minimal choice of $m$ empirically.

\begin{figure}
\centering
\includegraphics[width=0.8\textwidth]{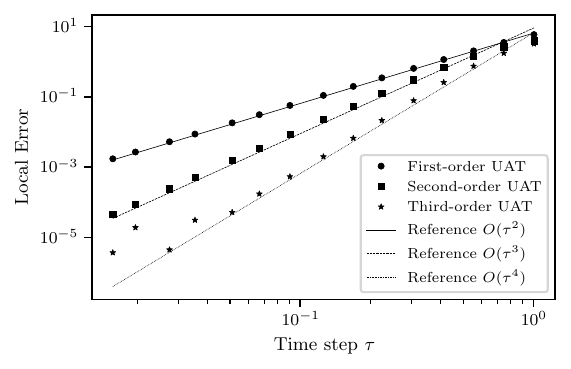}
\caption{Scaling of the local error with the time step $\tau$.}
\label{f.1}
\end{figure}

In Figure~\ref{f.1}, we confirm numerically the theoretical
convergence rate with respect to $\tau$ for the first, second and
third order schemes given by \eqref{lthOrdScheme}.  Shown is the local error
\begin{equation}
 E_{\text{loc}}(\tau) = \|\phi_1 -\phi(\tau)\| \,,
\end{equation}
which corresponds to $i=1$ in \eqref{e.global}, for fixed $c=200$ as
the time step $\tau$ is varied.  As explained in
Section~\ref{Sec.AccuSchm}, we work with time steps that are integer
multiples of the fast period, i.e., $\tau = \frac{2 \pi}{c^2} k$ for
integer $k$.  It is possible to modify the code such that arbitrary
time steps are possible.  However, this would require retaining all
factors $\e^{\pm c^2 t_i}$ in the generating formula
\eqref{e.DuhamelMainODE3} and all expressions that follow, and the
second sum in \eqref{SchemeSeq3} would already appear at the top level
of the recursion, cf.\ Remark~\ref{r.tau}.  As there is no advantage
of doing so, we did not implement this general case.

Figure~\ref{f.1} shows, in particular, that the local error of the
third order method scales like $\tau^{4}$, thus the global error will
scale like $\tau^3$, except for rather small values of $\tau$ where
the limitations of double-precision floating point begin to matter.
In general, floating-point errors might occur when we use the
quadrature formula on very small intervals.  As Figure~\ref{f.1}
shows, this occurs when using the third order scheme for small values
of $\tau$, since calling the function $\Psi_3(w, \tau)$ defined in
\eqref{SchemeSeq3.b} includes implicit calls of $\Psi_2(w, y)$ and
$\Psi_1(w, x)$ for $x,y$ with $x \ll y\ll \tau$, which means that the
lengths of the subintervals on which we use the Gauss formula for the
discrete sums when calling $\Psi_2(w, y)$ and $\Psi_1(w, x)$ are very
small.

\begin{figure}
\centering
\includegraphics[width=0.8\textwidth]{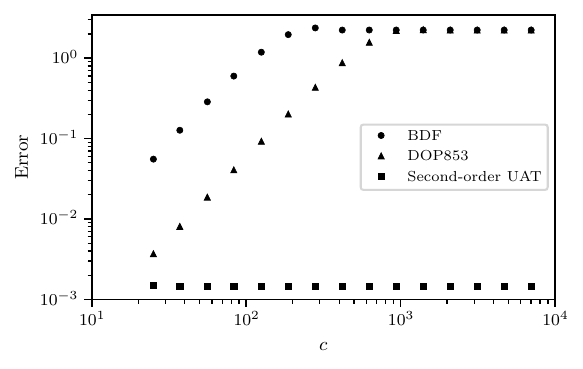}
\caption{Comparison of the second-order uniformly accurate time
integrator with two of the standard solvers from Scipy: the implicit
variable-order backward differentiation scheme and the explicit
Dormand--Prince embedded order 8(5,3) Runge--Kutta scheme.}
\label{f.2}
\end{figure}

\begin{figure}
\centering
\includegraphics[width=0.8\textwidth]{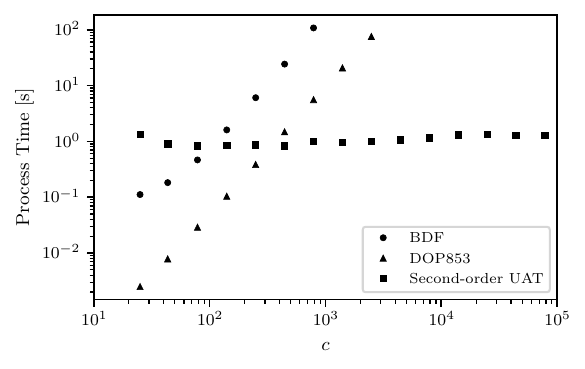}
\caption{Computation time for a single time step of the uniformly
accurate time integrator vs.\ the computation time over the same
interval of time using the built-in solvers, controlling their error
tolerances so that their error is no more than 30\% larger than the
error of the uniform scheme.}
\label{f.3}
\end{figure}

Figure~\ref{f.2} illustrate the uniformity of the error as a function
of $c$ using the second order UAT integrator compared with the
explicit Dormand--Prince embedded order 8(5,3) Runge--Kutta
(``DOP853'') scheme \cite{HNW:1993:SolvODE} and the implicit
multi-step variable-order (1 to 5) method based on a backward
differentiation formula (``BDF'') for the derivative approximation
\cite{Gear:1967:TheNI}.  Error performance of the new scheme is
uniform, while the error increases with $c$ for the built-in schemes,
indicating that their error indicator heuristics are insufficient for
dealing with such extreme multi-scale dynamics.
 
In Figure~\ref{f.3}, we compare the computation time for single time
step using the second order integrator with the solvers used in
Figure~\ref{f.2}.  We see that computation time goes up quadratically
in $c$, as expected, while computation time of the UAT scheme is
constant by design.  Timings refer to our reference implementation in
Python, which is available as supplementary material, on a single core
of an Intel i7 mobile processor, without any attempt at
speed-optimizing the code which is bottlenecked in the Python
interpreter for this low-dimensional test problem.  A more involved
study on approximate slow manifolds for semilinear equations of
Klein--Gordon type is current work-in-progress and will be reported on
separately.

\section*{Acknowledgments}

The work was supported by German Research Foundation (DFG) grants
OL-155/6-2 and MO-4162/1-1.  The authors acknowledge additional
support through German Research Foundation Collaborative Research
Center TRR 181 under project number 274762653.

\bibliographystyle{siam}
\bibliography{kg}

\end{document}